\documentclass[11pt,notitlepage,a4paper,reqno,ps]{amsart}

\usepackage{amstext}
\usepackage{hyperref}

\usepackage{paralist}

\usepackage{xfrac,stmaryrd,mathrsfs,bm,amsthm,mathtools,yfonts,amssymb,esint}
\usepackage{color}

\newcommand{\medint}{-\kern  -,375cm\int}

\theoremstyle{plain}
\newtheorem{theorem}{Theorem}[section]
\newtheorem{corollary}[theorem]{Corollary}
\newtheorem{lemma}[theorem]{Lemma}
\newtheorem{proposition}[theorem]{Proposition}
\theoremstyle{definition}

\theoremstyle{remark}
\newtheorem{remark}[theorem]{Remark}
\theoremstyle{plain}

\def\eps{\varepsilon}

\def\div{{\rm div}}

\newcommand{\R}{\mathbb R}
\newcommand{\Rn}{\mathbb{R}^n}

\newcommand\res{\mathop{\hbox{\vrule height 7pt width .5pt depth 0pt
\vrule height .5pt width 6pt depth 0pt}}\nolimits}
\newcommand{\cL}{\mathcal{L}}
\newcommand{\Ln}{\mathcal{L}^n}

\newcommand{\Om}{\Omega}

\newcommand{\A}{\mathbb {A}}

\newcommand{\EEE}{\mathscr{E}}

\newcommand{\de}{\partial}

\newcommand{\Lip}{{\rm {Lip}}}

\renewcommand{\div}{\textup{div}}

\newcommand\Sing{\textup{Sing}}
\newcommand\Reg{\textup{Reg}}

\newcommand{\aaa}{{\bf a}}

\newcommand{\graph}{\textup{graph}}

\numberwithin{equation}{section} \makeatletter

\renewcommand{\p@enumi}{\thesection.}
\makeatother  \allowdisplaybreaks

\begin{document}
\title[Free boundary regularity for nonlinear obstacle problems]{
The classical obstacle problem\\ for nonlinear variational energies}

\author[M. Focardi, F. Geraci, E. Spadaro]{M. Focardi, F. Geraci, E. Spadaro}

\address{DiMaI
``U. Dini'', Universit\`a di Firenze,
V.le Morgagni 67/A, I-50134 Firenze}
\email{focardi@math.unifi.it, geraci@math.unifi.it}

\address{Max-Planck-Institut f\"ur Mathematik in den Naturwissenschaften, 
Inselstrasse 22, D-04103 Leipzig} 
\email{spadaro@mis.mpg.de}

\dedicatory{{Dedicated to Nicola Fusco, a mentor and a friend.}}

\begin{abstract}
We develop the complete free boundary analysis for solutions to classical obstacle problems related to nondegenerate 
nonlinear variational energies. The key tools are optimal $C^{1,1}$ regularity, which we review more generally 
for solutions to variational inequalities driven by nonlinear coercive smooth vector fields, and the 
results in \cite{FocGelSp15} concerning the obstacle problem for quadratic energies with Lipschitz coefficients. 

Furthermore, we highlight similar conclusions for locally coercive vector fields having in mind applications 
to the area functional, or more generally to area-type functionals, as well.
\end{abstract}

\maketitle

\section{Introduction}

Variational inequalities are a classical topic in partial differential equations starting with the seminal 
works of Fichera and Stampacchia in the early 60's, motivated by a wide variety of applications in mechanics
and other applied sciences. 
This subject has been developed over the last 50 years by the works of many authors; it is not realistic to 
give here a complete account: we rather refer to the books and surveys 
\cite{BaCa84,CS,Duvaut-Lions,Frie,KS,PSU,Ro,Troianiello,Uraltseva87} 
for a fairly vast bibliography and its historical developments.

To introduce the problem, let $\psi$ and $g$
be given functions in $W^{1,p}(\Omega)$, $p\in(1,\infty)$, 
with $g\geq \psi$ $\Ln$ a.e. on $\Omega$  and set
\begin{equation}\label{e:Kpsi}
{\mathbb{K}_{\psi,g}}:=\{v\in g+W^{1,p}_0(\Omega):\,v\geq\psi\quad\text{$\Ln$ a.e. on $\Omega$}\}.        
\end{equation}
Consider a smooth \emph{coercive vector field} $(a_0,\aaa):\Omega\times\R\times\R^n\to\R\times\R^n$ 
according to \cite[Definition~3.1 of Chapter~IV]{KS} and \cite[Chapter~4]{Troianiello} (cf. 
Section~\ref{s:cvf} for the precise definitions and the necessary assumptions). 
The existence of a solution $u\in {\mathbb{K}_{\psi,g}}$ of the problem
\begin{equation}\label{e:mainpb}
\int_\Omega\aaa(x,u,\nabla u)\cdot\nabla(v-u)\,dx+\int_\Omega a_0(x,u,\nabla u)(v-u)dx\geq 0
\qquad\text{for all $v\in {\mathbb{K}_{\psi,g}}$,}
\end{equation}
is well-known (cf. \cite[Section~4 of Chapter~III]{KS} if $p=2$ and \cite[Chapter~4]{Troianiello} otherwise)
{and shortly recalled in Section~\ref{s:cvf} below}. 
Under suitable hypotheses on the fields, classical results ensure optimal regularity for $u$, i.e. 
$u\in C^{1,1}_{loc}(\Omega)$, as long as $\psi\in C^{1,1}_{loc}(\Omega)$ (cf. for instance 
\cite[Sections~4.5-4.6]{Troianiello} in the quadratic case, and \cite{Uraltseva87} in general).

The prototype example we have in mind is that of nonlinear variational problems
\begin{equation}\label{e:mainpb2}
\min_{v\in{\mathbb{K}_{\psi,g}}}\int_\Omega F(x,v,\nabla v)\,dx
\end{equation}
that leads to a variational inequality of the form \eqref{e:mainpb} with $\aaa=\nabla_\xi F$ and $a_0=\partial_z F$,
under suitable assumptions on $F=F(x,z,\xi)$ such as global smoothness, convexity and $p$-growth in the last 
variable (cf. Theorem~\ref{t:main} below for the precise assumptions on $F$). 

The aim of this short note is to perform an exhaustive analysis of the \emph{free boundary}, i.e.~the set 
$\partial\{u=\psi\}$, for the broad class of obstacle problems introduced in \eqref{e:mainpb2}, 
and to establish a parallel with the known results in the quadratic case as developed by Caffarelli \cite{Ca98}, 
Weiss \cite{Weiss} and Monneau \cite{Monneau} (cf. Theorem~\ref{t:main} for the statement).

The sharp analysis and stratification of the free boundary we provide is an outcome of 
{a suitable linearization argument (cf. Lemma~\ref{l:tecnico2} below) and of} the analogous results 
for the classical obstacle problem for quadratic energies with Lipschitz coefficients recently proved in 
\cite{FocGelSp15} that we state in Section~\ref{s:prel} (cf. Theorem~\ref{t:mainFGS}). It corresponds to the case
$F(x,\xi)=\A(x)\xi\cdot\xi$ in \eqref{e:mainpb2}, with $\A\in \mathrm{Lip}(\Om,\R^{n\times n})$ 
defining a continuous and coercive quadratic form.
The lack of smoothness and homogeneity of the matrix of coefficients $\A$ in Theorem~\ref{t:mainFGS}
does not permit to exploit elementary freezing arguments to locally reduce the regularity 
problem above to the analogous one for smooth operators, for which a complete theory has 
been developed by Caffarelli in a long term program \cite{Ca77,Ca80,Ca98-Fermi,Ca98}. 
Building upon the variational approach to the classical obstacle problem developed by Weiss 
\cite{Weiss} and Monneau \cite{Monneau}, the strategy to prove Theorem~\ref{t:mainFGS} is 
energy-based and relies on quasi-mononicity formulas extending those of  Weiss \cite{Weiss} 
and Monneau \cite{Monneau}, on Weiss's epiperimetric inequality as well as on Caffarelli's 
fundamental blow up analysis \cite{Ca80}.

As a direct outcome of Theorem~\ref{t:mainFGS} we shall deduce the analogous result for solutions of 
\eqref{e:mainpb2} (cf. Theorem~\ref{t:main}). Furthermore, adding suitable assumptions on the data of 
the problem, we can provide similar conclusions in case the vector field $\nabla_\xi F$ is more generally 
\emph{locally coercive}, thus including in our analysis the important case of the area functional.

A short summary of the contents of the paper is resumed in what follows: Section~\ref{s:prel} is devoted
to fix the notation and state the conclusions of the free boundary analysis in the quadratic case 
following \cite{FocGelSp15}.
In Section~\ref{s:cvf} we introduce the necessary definitions to state the main result 
of the paper, Theorem~\ref{t:main},  and show how the latter follows directly from Theorem~\ref{t:mainFGS}. In doing this, we shall 
first review almost optimal and then optimal regularity in the broader setting of solutions to variational 
inequalities driven by coercive vector fields as in \eqref{e:mainpb} (cf. Theorems~\ref{t:noptreg}
and \ref{t:optreg}), and then develop in details the analysis of the free boundary in the variational case 
in \eqref{e:mainpb2}. 
Finally, in Section~\ref{s:lcvf} we highlight the required changes to deduce similar conclusions for the 
case of locally coercive vector fields, and also analyze the case of the area functional in a Riemannian
manifold.

Non-optimal regularity for solutions is a classical topic well-known in literature at least in the quadratic 
case $p=2$ that has been established in several fashions: by penalization methods (cf. \cite{LewyStampacchi69}, 
\cite{BrezisStampacchia68}, \cite{Brezis72}), by Lewy-Stampacchia inequalities (cf. \cite{MoscoTroianiello73}, 
\cite{Mosco79}, \cite{HanouzetJoly79}, \cite{FrehseMosco82}, \cite{Troianiello}), by local comparison methods 
(cf. \cite{Giaq81}), by introducing a substitute variational inequality (cf. \cite{Gustaffson86}), 
{and by the linearization method (see \cite{Fu90, FuMin00}).
By following the streamline of ideas of the latter technique introduced in \cite{Fu90}, we provide here an elementary variational proof valid in the 
general framework of nonlinear variational inequalities under investigation.}
In particular, we show that solutions of \eqref{e:mainpb} satisfies a nonlinear elliptic PDE in divergence form, 
in turn from this suboptimal regularity can be established (for further comments cf. Section~2). 

Finally, we are able to establish optimal regularity following Gerhardt \cite{Ger73} {(see \cite{Frehse72,BrezKinder73,CaffaKind80} for the classical results).}
In addition, we remark that solutions to \eqref{e:mainpb} are actually \emph{$Q$-minima} of a related 
functional according to Giaquinta and Giusti~\cite{GiaqGiu82,GiaqGiu84}.

Furthermore, in the case of the area functional one can prove that solutions to the obstacle problem are actually 
almost minimizers of the perimeter, thus leading by a well-known theory of minimal surfaces (cf.~\cite{Tamanini84}) 
to estimates on the gradient of the solutions which bypass the global approach by Hartman and 
Stampacchia~\cite{HartmanStampacchia66} exploiting the bounded slope condition and the construction of barriers.
\vskip0.25cm

To conclude this introduction M.F. would like to add some personal annotations. I had the luck of attending a 
PhD course on Calculus of Variations taught by Nicola Fusco when I was still a graduate student in Florence 
trying to find my way through Mathematics. I clearly remember Nicola's mastery of the subject, his enthusiasm 
in transmitting the beauty of many ideas, and the pleasant atmosphere in the classroom despite several difficult 
proofs and lengthy calculations. That course pushed my interest forward Calculus of Variations. The influence of 
Nicola on my professional life is still active nowadays: on one hand in a direct way having the possibility 
to collaborate with him, and on the other hand indirectly in studying and exploiting many important results of his.
All this written, it is a great pleasure for me to contribute with this note to celebrate Nicola's birthday.

\section{Preliminaries}\label{s:prel}

The scalar product in $\Rn$ is denoted by $\xi\cdot\eta$ for all $\xi$, $\eta\in\Rn$, while 
$\langle\cdot,\cdot\rangle$ is generically used to indicate a duality pairing of the relevant 
function spaces. We use standard notation for Lebesgue and Hausdorff measures, for Lebesgue and Sobolev spaces.

With $c$ we denote a positive constant that may vary from line to line, we shall always highlight 
the parameters on which the constant depends.

We state explicitly only the ensuing result since it will be instrumental for our purposes.
\begin{theorem}[Theorem~1.1 \cite{FocGelSp15}]\label{t:mainFGS}
Let $\Om\subset\Rn$ be smooth, bounded and open;
$\A\in \mathrm{Lip}(\Om,\R^{n\times n})$ be symmetric and uniformly elliptic,
i.e.~$\lambda^{-1}|\xi|^2\leq \A(x) \xi\cdot\xi\leq\lambda|\xi|^2$
for all $x\in \Omega$ and all $\xi \in \R^n$ for some $\lambda\geq1$; $f\in C^{0,\alpha}(\Om)$ for some 
$\alpha\in(0,1]$; $g\in W^{1/2,2}(\partial\Om)$; 
$\psi\in C^{1,1}_{loc}(\Omega)$ such that $\psi\leq g$ $\mathcal{H}^{n-1}$-a.e on 
$\partial\Omega$ with $\div(\A\nabla\psi)\in C^{0,\alpha}(\Om)$ in the sense 
of distributions and with $f - \div(\A\nabla\psi)\geq c_0 >0$ {for some constant $c_0$}.

Let $u$ be the (unique) minimizer of 
\[
\EEE[v]:=\int_\Omega\big(\A(x)\nabla v(x)\cdot\nabla v(x)
+2 f(x)\,v(x)\big)\,dx,
\]
on the set ${\mathbb{K}_{\psi,g}}$ introduced in \eqref{e:Kpsi} (with $p=2$).

Then, $u$ is $C^{1,{\tau}}_{loc}(\Om)$ for every $\tau\in (0,1)$, and the free boundary 
decomposes as $\de \{u = \psi\} \cap \Omega = \Reg(u) \cup \Sing(u)$, where 
$\Reg(u)$ and $\Sing(u)$ are called its regular and singular part, respectively. 
Moreover, $\Reg(u) \cap \Sing(u) = \emptyset$ and
\begin{itemize}
\item[(i)] $\Reg(u)$ is relatively open in $\de \{u = \psi\}$
and, for every point $x_0 \in \Reg(u)$, there exist $r=r(x_0)>0$
and $\beta= \beta(x_0) \in (0,1)$ such that
$\Reg(u) \cap B_r(x_0)$ is a $C^{1,\beta}$ submanifold of dimension $n-1$;
\item[(ii)] $\Sing(u) = \cup_{k=0}^{n-1} S_k$, with $S_k$ contained in the union of
at most countably many submanifolds of dimension $k$ and class $C^1$.
\end{itemize}
\end{theorem}
\begin{remark}
 Following the generalization of the previous result by the second named author in \cite{Geraci16}, 
 we can actually require $f$ above to satisfy only a suitable Dini-type continuity condition to 
 conclude an analogous free boundary analysis.
\end{remark}

\section{Coercive vector fields}\label{s:cvf}

Let $\Om\subset\Rn$ be a smooth, bounded and open set. 
Consider $(a_0,\aaa):\Omega\times\R\times\R^n\to\R\times\R^n$ a smooth vector field satisfying
(cf. \cite[Section~4.3.2]{Troianiello})
\begin{itemize}
\item[(H1)] $a_0$ is Carath\'eodory, 
$\aaa\in C^{1,1}_{loc}(\Omega\times\R\times\R^n,\R^{n})$ and there is $p\in(1,\infty)$, for which 
\begin{itemize}
\item[(i)] $\big(\aaa(x,z,\xi)\cdot\xi\big)\wedge\big(a_0(x,z,\xi)z\big)
\geq \lambda|\xi|^p+\lambda_1|z|^p-\phi_1(x)$ for $\Ln$ a.e. 
$x\in\Omega$, and for all $z\in\R$, $\xi\in\Rn$, with  $\phi_1\in L^1(\Omega)$, $\lambda>0$ and $\lambda_1\geq0$;
\item[(ii)] $|a_0(x,z,\xi)|\vee|\aaa(x,z,\xi)|\leq \Lambda(|z|^{p-1}+|\xi|^{p-1})+\phi_2(x)$  for $\Ln$ a.e. 
$x\in\Omega$ and for all $(z,\xi)\in\R\times\R^n$, with $\Lambda>0$ and $\phi_2\in L^{\frac{p}{p-1}}(\Omega)$;
\item[(iii)] there is a constant $\Theta>0$  such that for all $x\in\Omega$, $z,\,\zeta\in\R$, and 
$\xi\in\R^n$
\[
 |\aaa(x,z,\xi)-\aaa(x,\zeta,\xi)|\leq \Theta|z-\zeta|(1+|\xi|^{p-1});
\]
\end{itemize}
\item[(H2)]  
for $\Ln$ a.e. $x\in\Omega$, and for all $z\in\R$, $\xi,\,\eta\in \R^n$
\begin{equation}\label{e:amonotone}
0\leq \big(\aaa(x,z,\xi)-\aaa(x,z,\eta)\big)\cdot(\xi-\eta),
\end{equation}
with strict inequality sign for $\xi\neq\eta$.
\end{itemize}
Note that strongly coercive vector fields as defined in \cite[Definition~3.1 of Chapter~IV]{KS} 
satisfy the assumptions above.

Under conditions (H1)-(H2) and supposing the obstacle $\psi$ and the boundary datum $g$ in $W^{1,p}(\Omega)$
and satisfying the compatibility condition $g\geq\psi$ $\cL^n$ a.e. on $\Omega$, the existence of solutions to 
\eqref{e:mainpb} is a consequence of classical results.
Indeed, consider the nonlinear operator $\mathscr{A}:W^{1,p}(\Omega)\mapsto W^{1,-p^\prime}(\Omega)$ defined by
\begin{equation}\label{e:Aoperator}
\langle \mathscr{A}(w),v\rangle:=\int_\Omega\big(\widetilde{\aaa}(x,w,\nabla w)\cdot \nabla v
+ \widetilde{a}_0(x,w,\nabla w)\,v\big)\,dx
\end{equation}
for $w\in W^{1,p}(\Omega)$ and $v\in W^{1,p}_0(\Omega)$, where for all $(x,z,\xi)\in\Omega\times\R\times\R^n$
\[
 \widetilde{\aaa}(x,z,\xi):=\aaa(x,z+g(x),\xi+\nabla g(x)),
\qquad \widetilde{a}_0(x,z,\xi):={a_0}(x,z+g(x),\xi+\nabla g(x)).
\]
Note that $\widetilde{\aaa}$ and $\widetilde{a}_0$ are Carath\'eodory functions on account of the regularity
of $\aaa$ and $a_0$.
Then, items (i) and (ii) in (H1) yield that $\mathscr{A}$ is coercive relative to the closed (in the norm 
topology of $W^{1,p}$) convex subset $\mathbb{K}_{\psi-g,0}$ of 
$W^{1,p}_0(\Omega)$ given by  
\[
{\mathbb{K}}_{\psi-g,0}:=
\{v\in  W^{1,p}_0(\Omega):\,v\geq\psi-g\quad\text{$\Ln$ a.e. on $\Omega$}\}.
\]
More precisely, for some $w_0\in {\mathbb{K}}_{\psi-g,0}$ (actually for any $w_0$ in this case)
\[
 \lim_{w\in W^{1,p}_0(\Omega),\,\|w\|_{W^{1,p}(\Omega)}\to\infty}\|w\|^{-1}_{W^{1,p}(\Omega)}
 \langle\mathscr{A}(w),w-w_0\rangle=+\infty.
 \]
\begin{remark}
Coercivity is clearly ensured under weaker conditions than those in item (i) of (H1) in view of Sobolev embedding theorems (cf. \cite[Theorems 3.7 and 3.8]{Gi03})
\end{remark}

In particular, \cite[condition (4.26)]{Troianiello} is fulfilled for any $w_0\in {\mathbb{K}}_{\psi-g,0}$ and for any $R>0$. 
Since the injection $W^{1,p}(\Omega)\hookrightarrow L^p(\Omega)$ is compact, assumption (H2) gives that $\mathscr{A}$ 
is a Leray-Lions operator (cf. \cite[Theorem~4.21]{Troianiello}). Existence of a solution 
$\widetilde{u}\in {\mathbb{K}}_{\psi-g,0}$ for 
\[
\int_\Omega\widetilde{\aaa}(x,\widetilde{u},\nabla \widetilde{u})\cdot\nabla(v-\widetilde{u})\,dx
+\int_\Omega \widetilde{a}_0(x,\widetilde{u},\nabla \widetilde{u})(v-\widetilde{u})dx\geq 0
\qquad\text{for all $v\in {\mathbb{K}}_{\psi-g,0}$}
\]
follows at once from \cite[Lemma~4.13 and Theorem~4.17]{Troianiello}.
Therefore, $u:=\widetilde{u}+g$ is a solution to \eqref{e:mainpb}.

Finally, uniqueness of solutions to \eqref{e:mainpb} is guaranteed in case the ensuing more stringent 
monotonicity condition is satisfied 
\begin{equation}\label{e:Tmonotone}
 0\leq \big(\aaa(x,z,\xi)-\aaa(x,\zeta,\eta)\big)\cdot(\xi-\eta)+
 \big(a_0(x,z,\xi)-a_0(x,\zeta,\eta)\big)(z-\zeta),
 \end{equation}
for $\Ln$ a.e. $x\in\Omega$, for all $z,\,\zeta\in\R$ and $\xi,\,\eta\in\R^n$, with strict inequality sign 
in \eqref{e:Tmonotone} if $\xi\neq\eta$.
Disregarding the characterization of the equality case in \eqref{e:Tmonotone}, the latter condition 
yields that the nonlinear operator $\mathscr{A}$ defined in \eqref{e:Aoperator} is monotone, actually 
$T$-monotone (cf. \cite[p.~231]{Troianiello}).

 In the variational case in which $\aaa=\nabla_\xi F$ and $a_0=\partial_z F$, (H2) follows from the convexity 
 of the Lagrangian $F$ in the gradient variable $\xi$, while \eqref{e:Tmonotone} from the joint convexity of
 $F$ in $(z,\xi)$.

\subsection{Regularity of solutions}

In what follows we consider variational inequalities as in \eqref{e:mainpb} for vector fields $(a_0,\aaa)$ 
satisfying (H1)-(H2) and further assuming the following conditions on the {obstacle function}: 
\begin{itemize}
\item[(H3)] $\psi\in C^{1,1}_{loc}(\Omega)$. 
\end{itemize}
{Note then that 
\begin{equation}\label{e:h}
h:=-\div\big(\aaa(x,\psi,\nabla\psi)\big)+a_0(x,\psi,\nabla\psi)\in L^\infty_{loc}(\Omega).
\end{equation}
}
The key to establish optimal regularity is contained in Proposition~\ref{p:tecnico1} in which we 
switch from a variational inequality to a nonlinear elliptic PDE in divergence form. Indeed, 
on account of Proposition~\ref{p:tecnico1}, in Theorem~\ref{t:noptreg} we establish almost optimal 
regularity of solutions through classical elliptic regularity results and finally optimal regularity 
is achieved in Theorem~\ref{t:optreg} by means of Gerhardt's approach (cf. \cite{Ger73}). 

Despite almost optimal regularity of solutions is a well-studied subject, we provide in 
Proposition~\ref{p:tecnico1} and Theorem~\ref{t:noptreg} below a different proof 
that departs from the classical ones known in literature (\cite{LewyStampacchi69,BrezisStampacchia68,Brezis72,
MoscoTroianiello73,HanouzetJoly79,Giaq81,FrehseMosco82,Gustaffson86,Troianiello})
{by extending the linearization method to the general setting studied here (cf. \cite{Fu90,FuMin00}).}
The idea is to reduce regularity for variational inequalities of the sort in \eqref{e:mainpb} to the more standard setting of nonlinear elliptic PDEs. 
In the case of quadratic forms a similar argument has been established in \cite{FocGelSp15} for the obstacle problem in Theorem~\ref{t:mainFGS}, inspired by the case discussed in \cite{Weiss} for the Laplacian. 
\begin{proposition}\label{p:tecnico1}
Let (H1)-{(H3)} hold true. Then, a solution $u\in {\mathbb{K}_{\psi,g}}$ to problem \eqref{e:mainpb} 
solves  { 
\begin{equation}\label{e:pde}
-\div(\aaa(x,u,\nabla u))+a_0(x,u,\nabla u)=\zeta(x)
\end{equation}
$\Ln$ a.e. in $\Omega$ and in $\mathcal{D}^\prime(\Omega)$, for some 
function $\zeta\in L^\infty_{loc}(\Omega)$ such that, for $h$ defined in \eqref{e:h},
\[
0\leq \zeta\leq 
h^+\,\chi_{\{u=\psi\}}\quad\textrm{$\Ln$ a.e. in $\Omega$}.
\]}
\end{proposition}
\begin{proof} 
 Let $\varphi\in C^\infty_c(\Omega)$ and for all $\eps>0$ take 
 $v_\eps:=(u+\eps\varphi)\vee\psi\in {\mathbb{K}_{\psi,g}}$ as test function in \eqref{e:mainpb}.
 Note that in case $\varphi$ is a non-negative function we obtain
\begin{equation}\label{e:mainpb4}
 \int_\Omega\aaa(x,u,\nabla u)\cdot\nabla \varphi\,dx+\int_\Omega a_0(x,u,\nabla u)\varphi\, dx\geq 0.
 \end{equation}
 Therefore, the distributional divergence $\div(\aaa(\cdot,u,\nabla u))$ of $\aaa(\cdot,u,\nabla u)$ 
 satisfies
\begin{equation*}
\langle -\div(\aaa(\cdot,u,\nabla u))+a_0(\cdot,u,\nabla u)\Ln\res\Omega,\varphi\rangle\geq 0\qquad
   \text{for all $\varphi\in C^\infty_c(\Omega)$, $\varphi\geq 0$},
\end{equation*}
in turn implying that $\mu:=-\div(\aaa(\cdot,u,\nabla u))+a_0(\cdot,u,\nabla u)\Ln\res\Omega$ is a 
non-negative Radon measure. 

Next, consider $v_\eps$ as above with no sign assumptions on $\varphi$, set $\Omega_\eps:=\{u+\eps\varphi<\psi\}$,
and rewrite the two addends in \eqref{e:mainpb} respectively as follows
\begin{equation*}
 \int_\Omega\aaa(x,u,\nabla u)\cdot\nabla(v_\eps-u)dx=\eps \int_\Omega\aaa(x,u,\nabla u)\cdot\nabla\varphi\,dx
 +\int_{\Omega_\eps}\aaa(x,u,\nabla u)\cdot\nabla\big(\psi-(u+\eps\varphi)\big)dx,
\end{equation*}
and
\begin{equation*}
 \int_\Omega a_0(x,u,\nabla u)(v_\eps-u)dx=
 \eps\int_\Omega a_0(x,u,\nabla u)\varphi\,dx
 +\int_{\Omega_\eps}a_0(x,u,\nabla u)\big(\psi-(u+\eps\varphi)\big)dx.
\end{equation*}
Thus, on account of the definition of the measure $\mu$ we conclude that
\begin{equation*}
\eps\int_\Omega\varphi\,d\mu\geq -\int_{\Omega_\eps}\aaa(x,u,\nabla u)\cdot\nabla\big(\psi-(u+\eps\varphi)\big)dx
-\int_{\Omega_\eps}a_0(x,u,\nabla u)\,\big(\psi-(u+\eps\varphi)\big)dx.
\end{equation*}
By the monotonicity hypothesis on the field $\aaa$ in (H2) we have that
\begin{multline*}
\eps\int_\Omega\varphi\,d\mu
\geq-\int_{\Omega_\eps}\aaa(x,u,\nabla \psi)\cdot\nabla\big(\psi-u\big)dx\\
+\eps\int_{\Omega_\eps}\aaa(x,u,\nabla u)\cdot\nabla\varphi\,dx
-\int_{\Omega_\eps}a_0(x,u,\nabla u)\big(\psi-(u+\eps\varphi)\big)dx
\end{multline*}
and therefore we infer that 
\begin{multline}\label{e:variations}
\eps\int_\Omega\varphi\,d\mu\geq 
\underbrace{-\int_{\Omega_\eps}\Big(\aaa(x,\psi,\nabla \psi)\cdot\nabla\big(\psi-(u+\eps\varphi)\big)
+a_0(x,\psi,\nabla \psi)\big(\psi-(u+\eps\varphi)\big)\Big)dx}_{=:I^{(1)}_\eps}\\
+\underbrace{\eps\int_{\Omega_\eps}\big(\aaa(x,u,\nabla u)-\aaa(x,\psi,\nabla \psi)\big)\cdot\nabla\varphi\,dx
+\eps\int_{\Omega_\eps}\big(a_0(x,u,\nabla u)-a_0(x,\psi,\nabla \psi)\big)\varphi\,dx}_{=:I^{(2)}_\eps}\\
+\underbrace{\int_{\Omega_\eps}\big(\aaa(x,\psi,\nabla \psi)-\aaa(x,u,\nabla \psi))\cdot\nabla\big(\psi-u\big)dx
+\int_{\Omega_\eps}\big(a_0(x,\psi,\nabla \psi)-a_0(x,u,\nabla u)\big)\big(\psi-u)\big)dx}_{=:I^{(3)}_\eps}.
\end{multline}
We deal with the three terms above separately.  We start off with the first term that rewrites as
\[
  I^{(1)}_\eps=-\int_{\Omega}\Big(\aaa(x,\psi,\nabla \psi)\cdot\nabla\big((\psi-(u+\eps\varphi))\vee 0\big)
+a_0(x,\psi,\nabla \psi)\big((\psi-(u+\eps\varphi))\vee 0\big)\Big)dx.
\]
Being $u\ge\psi$ $\Ln$ a.e. in $\Omega$ and $\varphi\in C^\infty_c(\Omega)$, we have
$\Omega_\eps\subset\hskip-0.125cm\subset\Omega$, so that $(\psi-(u+\eps\varphi))\vee 0\in W^{1,p}_0(\Omega)$. 
By taking this into account, together with the condition $\psi\in C^{1,1}_{loc}(\Omega)$ (cf. (H3)), 
item (ii) in (H1) and an integration by parts yield, {recalling that $h=-\div\big(\aaa(x,\psi,\nabla\psi)\big)+a_0(x,\psi,\nabla\psi)$,}
\begin{multline}\label{e:I1}
  I^{(1)}_\eps=\int_{\Omega}\big(\div(\aaa(x,\psi,\nabla \psi))-a_0(x,\psi,\nabla \psi)\big)
  \big((\psi-(u+\eps\varphi))\vee 0\big)dx\\
  ={-\int_{\Omega_\eps}h\,\big((\psi-(u+\eps\varphi)\big)\,dx\geq
  -\int_{\Omega_\eps}h^+\,\big(\psi-(u+\eps\varphi)\big)\,dx\geq\eps\int_{\Omega_\eps}h^+\,\varphi\, dx}
\end{multline}
{where in the last but one equality we have used that $\psi-(u+\eps\varphi)\geq 0$ $\Ln$ a.e. on 
$\Omega_\eps$ and in the last one that $u\geq\psi$ $\Ln$ a.e. on $\Omega$. In turn, the latter condition 
implies that 
\[
\Ln\big(\big(\{u=\psi\}\cap\{\varphi<0\}\big)\setminus\Omega_\eps\big)=
\Ln\big(\Omega_\eps\setminus\big(\{0\leq u-\psi\leq\eps\|\varphi\|_{L^\infty(\Omega)}\}\cap \{\varphi<0\}\big)\big)=0,
\]
so that $\chi_{\Omega_\eps}\to\chi_{\{u=\psi\}\cap\{\varphi<0\}}$ in $L^1(\Omega)$, for every $\varphi\in C^\infty_c(\Omega)$. Therefore, from \eqref{e:I1} we infer
\begin{equation}\label{e:I1b}
\liminf_{\eps\to 0^+}\eps^{-1}I^{(1)}_\eps\geq\int_{\{u=\psi\}\cap\{\varphi<0\}} h^+\,\varphi\,dx.
\end{equation}
}
In addition, by the Dominated convergence theorem and by the locality of the weak gradient, we conclude that 
for every $\varphi\in C^\infty_c(\Omega)$ 
\begin{align}\label{e:I2}
\lim_{\eps\to 0^+}\eps^{-1}I^{(2)}_\eps&=\int_{\{u=\psi\}\cap\{\varphi<0\}}
 \big(\aaa(x,u,\nabla u)-\aaa(x,\psi,\nabla \psi)\big)\cdot\nabla\varphi\,dx\notag\\
&\quad+
\int_{\{u=\psi\}\cap\{\varphi<0\}} \big(a_0(x,u,\nabla u)-a_0(x,\psi,\nabla \psi)\big)\varphi\,dx=0.
\end{align}
Finally, to deal with $I^{(3)}_\eps$ {we use item (iii) in (H1) to deduce that}
\[
 I^{(3)}_\eps\geq-\eps\Theta\|\varphi\|_{L^\infty(\Omega)}\int_{\Omega_\eps}(1+|\nabla\psi|^{p-1})|\nabla(\psi-u)|\,dx
 -\eps\|\varphi\|_{L^\infty(\Omega)}\int_{\Omega_\eps}|a_0(x,u,\nabla u)-a_0(x,\psi,\nabla \psi)|\,dx.
\]
Therefore, {by the quoted convergence of $\chi_{\Omega_\eps}$} and by the locality of the weak gradient, as 
{in \eqref{e:I1b}} and \eqref{e:I2}, we conclude that 
\begin{equation}\label{e:I3}
{\liminf_{\eps\to 0^+}\eps^{-1}I^{(3)}_\eps\geq 0.}
\end{equation}
Resuming, by \eqref{e:I1b}, \eqref{e:I2} and \eqref{e:I3}, passing to the limit as $\eps\downarrow 0^+$ in 
\eqref{e:variations} divided by $\eps>0$, we infer that
\[
 \int_\Omega\varphi\,d\mu\geq {\int_{\{u=\psi\}\cap\{\varphi<0\}} h^+\,\varphi\,dx.}
\]
By approximation (and by applying the argument above to $-\varphi$) we infer that for every 
$\varphi\in C^0_c(\Omega)$
\begin{equation*}
\int_{\{u=\psi\}\cap\{\varphi<0\}}{h^+}\,\varphi\,dx\leq
\int_\Omega\varphi\,d\mu\leq\int_{\{u=\psi\}\cap\{\varphi>0\}}{h^+}\,\varphi\,dx.
\end{equation*}
In turn, the latter inequalities imply that $\mu<\hskip-0.125cm<\Ln\res\Omega$. 
{Thus, if $\mu=\zeta \Ln\res\Omega$, with $\zeta\in L^1(\Omega)$, we infer that 
$0\leq\zeta\leq h^+\chi_{\{u=\psi\}}$ $\Ln$ a.e. $\Omega$, so that 
$\zeta\in L^\infty_{loc}(\Omega)$ by \eqref{e:h}.}


In conclusion, as by definition $\mu=-\div(\aaa(\cdot,u,\nabla u))+a_0(\cdot,u,\nabla u)\Ln\res\Omega$,  equation \eqref{e:pde} follows at once. 
\end{proof}

\begin{remark}
{One can prove that a solution $u$ of \eqref{e:mainpb} is a $Q$-minimum of a lower order perturbation of the $p$-Dirichlet energy from the conclusions of 
Proposition~\ref{p:tecnico1} as argued in \cite{GiaqGiu84} (cf. also \cite[Chapter~6]{Gi03}).} 
More precisely, let $\mathscr{G}:\mathcal{B}(\Omega)\times W^{1,p}(\Omega)\to[0,\infty)$ 
be 
\[
 \mathscr{G}(w,A):=\int_A G\big(x,w(x),\nabla w(x)\big)\,dx,
 \]
where $A\in\mathcal{B}(\Omega)$, the class of Borel subsets of $\Omega$, and 
$G:\Omega\times\R\times\R^n\to[0,\infty)$ is the Carath\'eodory integrand defined by 
 \[
  G(x,z,\xi):=|\xi|^p+|z|^p+|\nabla\psi(x)|^p+|\phi_2(x)|^{\frac p{p-1}}
  +|\phi_1(x)|+|a_0(x,u(x),\nabla u(x))|^{\frac{p}{p-1}}.
 \]
 Then, there is a constant $Q=Q(p,\lambda,\Lambda)>1$ such that  
 \begin{equation}\label{e:Qmin}
  \mathscr{G}(u,K)\leq Q\,\mathscr{G}(w,K)
 \end{equation}
 for all $w\in g+W^{1,p}_0(\Omega)$ such that $K:=\mathrm{spt}(w-u)\subset\hskip-0.125cm\subset\Omega$. 
 Note that $|a_0(\cdot,u(\cdot),\nabla u(\cdot))|^{\frac{p}{p-1}}\in L^{1}(\Omega)$ by item {(ii) in (H1)}.
 The direct methods for regularity introduced by Giaquinta and Giusti \cite{GiaqGiu82, GiaqGiu84} 
 imply that $u\in C^{0,\alpha}_{loc}(\Omega)$ for some $\alpha\in(0,1]$
{under suitable assumptions on $\phi_1$, $\phi_2$, $a_0$ and $p$ (cf. \cite{FuMin00} for instance).} 

Actually, we can establish \eqref{e:Qmin} a priori, directly from \eqref{e:mainpb} 
by taking the family of test functions $v=w\vee\psi$ with $w$ as above by means 
of items (i) and (ii) in (H1). 
  
 Finally, we recall that under the standing assumptions on $(\aaa,a_0)$ upper semicontinuity and approximate continuity of $\psi$ suffices to establish continuity of solutions (cf. \cite{MichaelZiemer86}).  
In particular, this shows that the sets $\{u>\psi\}$ and $\Omega_\eps$, $\eps>0$ suitable, in the proof of Proposition~\ref{p:tecnico1} are actually open.   
\end{remark}

We are now ready to deduce almost optimal regularity for solutions to \eqref{e:mainpb} from standard
elliptic regularity provided item (iii) in (H1) and (H2) are substituted by the more restrictive 
\begin{itemize}
\item[(iii)$^\prime$] there is a constant $\Theta>0$ such that for all $x,\,y\in\Omega$, $z,\,\zeta\in\R$
and $\xi\in\R^n$
\[
 |\aaa(x,z,\xi)-\aaa(y,\zeta,\xi)|\leq \Theta(|x-y|+|z-\zeta|)(1+|\xi|^{p-1})
\]
\item[(H2)$^{\prime}$] there is $\nu>0$ such that for $\Ln$ a.e. $x\in\Omega$, and for all 
 $z\in\R$, $\xi,\,\eta\in \R^n$
\begin{equation}\label{e:acoercive}
\nu^{-1}(1+|\xi|+|\eta|)^{p-2}\,|\xi-\eta|^2\leq 
\big(\aaa(x,z,\xi)-\aaa(x,z,\eta)\big)\cdot(\xi-\eta)\leq
\nu(1+|\xi|+|\eta|)^{p-2}\,|\xi-\eta|^2;
\end{equation}
\end{itemize}
On account of \eqref{e:pde} in Proposition~\ref{p:tecnico1} suboptimal regularity follows.
\begin{theorem}[Almost optimal regularity]\label{t:noptreg}
Let (H1) (with (iii)$^\prime$ in place of (iii)), (H2)$^{\prime}$ and (H3) 
hold true. Let $u\in  {\mathbb{K}_{\psi,g}}$ be a solution to problem \eqref{e:mainpb}, then 
$u\in W^{2,q}_{loc}\cap C^{1,\alpha}_{loc}(\Omega)$ for all $q\in[1,\infty)$ and all $\alpha\in(0,1)$.
\end{theorem}
\begin{proof}
By taking into account that $u$ solves \eqref{e:pde} (cf. Proposition~\ref{p:tecnico1}), classical elliptic 
regularity for nonlinear elliptic equations in divergence form yield that $u\in C^{1,\alpha}_{loc}(\Omega)$
for some $\alpha\in(0,1)$ (cf. \cite[Section~3]{Manfredi88}, \cite[Chapter~5]{ManfrediPhD}).

It is also classical to prove that $u\in W^{2,2}_{loc}(\Omega)$
(cf.~\cite[Chapter 4, Theorem 5.2]{LadUra}) and by differentiation, on account of the 
$C^{1,\alpha}_{loc}$ regularity already established and (H1)-(H2)$^\prime$, that the weak derivatives of 
$u$ satisfy a linear uniformly elliptic PDE with H\"older coefficients and right hand side being the divergence 
of a field in $L^\infty_{loc}(\Omega,\R^n)$. Therefore, we may apply standard $L^q$-regularity estimates 
(cf. \cite[Theorem~10.15]{Gi03}) to conclude that 
{$u\in W^{2,q}_{loc}\cap C^{1,\alpha}_{loc}(\Omega)$} for all $q\in[1,\infty)$ and all $\alpha\in(0,1)$.
\end{proof}
{
\begin{corollary}\label{c:density}
Under the assumptions of Theorem~\ref{t:noptreg} the function $\zeta$ in \eqref{e:pde}
of Proposition~\ref{p:tecnico1} actually equals $h^+\chi_{\{u=\psi\}}$ $\cL^n$ a.e. on
$\Omega$.
\end{corollary} 
\begin{proof}
By the $W^{2,q}$ regularity of $u$ and the $C^{1,1}_{loc}$ regularity of $\aaa$, one can compute the divergence 
in the definition of the measure $\mu$ and use the locality of weak derivatives to conclude.
\end{proof}
}
Optimal $C^{1,1}_{loc}$ regularity {of solutions} follows at once from Gerhardt's result~\cite{Ger73} 
provided $a_0$ is locally Lipschitz continuous. 
\begin{theorem}[Optimal regularity]\label{t:optreg} 
Let (H1) (with (iii)$^\prime$ in place of (iii)), (H2)$^{\prime}$ and (H3) 
hold true, and assume $g\in C^2(\overline{\Omega})$ with $\psi<g$ on $\partial\Omega$, and 
$a_0\in C^{0,1}_{loc}(\Omega\times\R\times\R^n,\R)$.

If $u\in {\mathbb{K}_{\psi,g}}$ is a solution to problem \eqref{e:mainpb}, then $u\in C^{1,1}_{loc}(\Omega)$.
\end{theorem}
\begin{proof}
The proof is essentially that of \cite{Ger73} despite the forcing term, i.e.~$a_0(\cdot,u,\nabla u)$ in our 
case, is not in $C^{0,1}$ as required in the statement there. Nevertheless, a careful inspection of 
that proof shows that the slightly weaker assumption $a_0(\cdot,u,\nabla u)\in W^{1,q}_{loc}(\Omega)$ 
for all $q\in[1,\infty)$ actually suffices (cf. formula (16) there). 
In our setting this property is an immediate outcome of the regularity hypothesis on $a_0$ 
and Theorem~\ref{t:noptreg} above.
\end{proof}
\begin{remark}
We point out that for $p\neq 2$ the study of degenerate fields $\aaa$ deserves additional efforts. 
Optimal regularity of solutions to \eqref{e:mainpb} with $\aaa(\xi)=|\xi|^{p-2}\xi$ and $a_0(x,z)=f(x)z$, 
$f\in L^{\infty}(\Omega)$, has been established only recently in \cite{AndLindShah} (cf. the bibliography 
there for more detailed references, {and also the results in \cite{Fu90}}). 
That paper deals also with the case $\psi\in C^{1,\beta}(\Omega)$, 
$\beta\in(0,1)$, that is not covered by our methods. More precisely, it is established there that solutions 
are $C^{1,\beta\wedge\sfrac 1{(p-1)}}_{loc}(\Omega)$, $\beta\in(0,1]$, and
actually $C^{1,\beta}_{loc}$ in the homogeneous setting $f\equiv 0$.

Building upon Proposition~\ref{p:tecnico1} and a careful analysis of the estimates in \cite[Chapter~5]{ManfrediPhD}
one can actually show that $u\in C^{1,\alpha}_{loc}(\Omega)$, for all $\alpha\in(0,\frac 1{p-1}]\cap(0,1)$ 
for fields satisfying (H1) and the degenerate analogue of (H2)$^\prime$.
\end{remark}

We end this subsection pointing out that the conclusions of Proposition~\ref{p:tecnico1} and Theorems~\ref{t:noptreg} 
and \ref{t:optreg} extend to the more general setting of fields $a_0$ satisfying the so called \emph{unnatural} 
growth conditions following the terminology of Giusti~\cite{Gi03} (cf. formula (6.15) there), of which item (ii) 
in (H1) is a simple instance.

This claim is also true in case $a_0$ satisfies the \emph{natural} growth conditions (cf. \cite[formula (6.18)]{Gi03}) 
provided bounded solutions are taken into account. Existence of such solutions is guaranteed for bounded obstacles 
and bounded boundary data, for instance.

\subsection{Free boundary regularity in the variational case}

We are now ready to state and prove the main result of the paper. From now on we restrict to the variational
case, in which $\aaa=\nabla_\xi F$ and $a_0=\partial_z F$ for suitable integrands $F$. 
We need to rephrase assumptions (H1), and (H2)$^{\prime}$ terms of the energy density $F$ itself. 
In passing we note that item (i) in (H1) is not needed provided $F$ satisfies suitable convexity and growth 
conditions in view of the Direct Method of the Calculus of Variations.
Indeed, item (i) in (H1) has been used only in the proof of existence of solutions to \eqref{e:mainpb}.
\begin{theorem}\label{t:main}
 Let $\Om\subset\Rn$ be smooth, bounded and open, and $p\in(1,\infty)$. Assume (H3) for $\psi$, and 
 $g\in C^2(\overline{\Omega})$ with $\psi<g$ on $\partial\Omega$.  

 Let $F\in C^{2,1}_{loc}(\Omega\times\R\times\R^n)$ be satisfying 
 \begin{equation}\label{e:Fgrowth}
  c_1|\xi|^p-\phi(x)\leq F(x,z,\xi)\leq c_2|\xi|^p+c_3|z|^{p^\ast}+\phi(x)\qquad 
 \end{equation}
 for all $z\in\R$, $\xi\in\R^n$, for $\Ln$ a.e. $x\in\Omega$, where $\phi\in L^1(\Omega)$, $c_1,\,c_2>0$ 
 and $c_3\geq 0$, and $p^\ast$ is the Sobolev exponent of $p$ (thus $p^\ast$ is any exponent if $p\geq n$).

 Suppose that items (ii), (iii)$^\prime$ in (H1) are satisfied by $\aaa=\nabla_\xi F$ and 
 $a_0=\partial_z F$, and in addition assume $F(x,z,\cdot)$ to be uniformly convex uniformly in $(x,z)$ 
 {w.r.to $\xi$}, i.e.~there exists $\nu>1$ such that for all $x\in\Omega$, $z\in\R$ and $\xi,\,\eta\in\R^n$
 \begin{equation}\label{e:Funcvx}
 \nu^{-1}(1+|\eta|)^{p-2}|\xi|^2\leq \nabla^2_\xi F(x,z,\eta)\xi\cdot\xi\leq \nu(1+|\eta|)^{p-2}|\xi|^2.
 \end{equation}
 Then, {the minimum problem in \eqref{e:mainpb2} has (at least) a solution $u\in \mathbb{K}_{\psi,g}$, 
 and, moreover, every solution belongs to $C^{1,1}_{loc}(\Omega)$.}  
 
 Let $u\in \mathbb{K}_{\psi,g}$ be a solution. If, moreover, {$\psi$ satisfies 
 \begin{itemize}
 \item[(H4)] for some constant $c_0>0$ we have for $\Ln$ a.e. on $\Omega$  
 \[
 h=-\div\big(\nabla_\xi F(x,\psi,\nabla\psi)\big)+\partial_zF(x,\psi,\nabla\psi)\geq c_0>0;
 \]
 \item[(H5)] for some $\alpha\in(0,1)$
 \begin{equation*}
 \div\big(\nabla_\xi F(\cdot,u,\nabla\psi))\in C^{0,\alpha}_{loc}(\Omega),
 \end{equation*}
\end{itemize}}
 then the free boundary decomposes as $\de\{u=\psi\}\cap\Omega=\Reg(u)\cup\Sing(u)$, where $\Reg(u)$ and $\Sing(u)$ 
 are called its regular and singular part, respectively. Moreover, $\Reg(u)\cap\Sing(u)=\emptyset$ and
\begin{itemize}
\item[(i)] $\Reg(u)$ is relatively open in $\de \{u = \psi\}$
and, for every point $x_0 \in \Reg(u)$, there exist $r=r(x_0)>0$
and $\beta= \beta(x_0) \in (0,1)$ such that
$\Reg(u) \cap B_r(x_0)$ is a $C^{1,\beta}$ submanifold of dimension $n-1$;
\item[(ii)] $\Sing(u) = \cup_{k=0}^{n-1} S_k$, with $S_k$ contained in the union of
at most countably many submanifolds of dimension $k$ and class $C^1$.
\end{itemize}
\end{theorem}

\begin{remark}
 In case $F=F(x,\xi)$ the structural conditions imposed on $F$, i.e. convexity and \eqref{e:Fgrowth},
 imply item (ii) in (H1) (cf. \cite[Lemma~5.2]{Gi03}). Therefore, besides uniform convexity, the only 
 nontrivial assumption on $F$ is (iii)$^\prime$ in (H1). In turn, the latter is clearly satisfied in the 
 autonomous case $F=F(\xi)$.
\end{remark}

{
\begin{remark}
 Assumption (H4) corresponds to the well-known concavity assumption on the obstacle function $\psi$ in 
 the case of the Laplacian, or better to the localized form of such a condition introduced in \cite{CaffRiv}.
 Simple examples show that (H4) is a necessary request to expect regular free boundaries.
 \end{remark}
}
\begin{remark}
 In view of the regularity assumptions on $F$ and the optimal regularity of $u$, assumption {(H5)} 
 is basically an hypothesis on the obstacle $\psi$ that can be enforced by assuming more regularity on $\psi$ 
 itself. For instance, it is implied by taking $\psi\in C^{2,\alpha}_{loc}(\Omega)$.
 
 Finally, {non trivial examples show that a qualified continuity hypothesis on the relevant operator calculated on the obstacle function, weaker than H\"older continuity 
 imposed in (H5), 
 is actually necessary to conclude free boundary} regularity already in the classical case of the Laplacian 
 {(cf. \cite{Blank, Monneau}).} 
 \end{remark}

To establish Theorem~\ref{t:main} we introduce the ensuing linearization; in this way we rewrite the PDE 
in \eqref{e:pde} as a locally uniform elliptic equation with suitable locally Lipschitz continuous matrix 
coefficients in case the gradient of the solution itself shares such a regularity.
\begin{lemma}\label{l:tecnico2}
Let (H1)-(H4) hold true, and let $u\in C^{1,1}_{loc}(\Omega)$ be a solution of \eqref{e:mainpb2}.
Then, there exists a {symmetric} matrix field $\A:\Omega\to\R^{n\times n}$ such that 
{
\begin{equation}\label{e:pde2}
 \div\big(\A(x)\nabla (u-\psi)\big)=\big(-\div(\nabla_\xi F(x,u,\nabla \psi))+\partial_zF(x,u,\nabla u)\big)\,\chi_{\{u>\psi\}}
\end{equation}}
 $\Ln$ a.e. in $\Omega$ and in $\mathcal{D}^\prime(\Omega)$; with $\A$ satisfying 
 \begin{itemize}
\item[(i)] $\A\in C^{0,1}_{loc}(\Om,\R^{n\times n})$, 
\item[(ii)] 
for all $K\subset\hskip-0.125cm\subset\Omega$ there is $\lambda_K\geq 1$ for which
\begin{equation}\label{e:Aprop}
\lambda^{-1}_K|\xi|^2\leq \A(x) \xi\cdot\xi\leq\lambda_K|\xi|^2
\quad\mbox{for all $x \in K$ and for all $\xi\in\R^n$}.
\end{equation}
\end{itemize}
\end{lemma}
\begin{proof}
We start off rewriting the Euler-Lagrange equation \eqref{e:pde} as follows {
\begin{equation}\label{e:pde3}
\div\big(\nabla_\xi F(x,u,\nabla u)-\nabla_\xi F(x,u,\nabla \psi)\big)=
 \big(-\div(\nabla_\xi F(x,u,\nabla \psi))+\partial_zF(x,u,\nabla u)\big)\,\chi_{\{u>\psi\}}.
\end{equation}}
In claiming the last equality we have used {Corollary~\ref{c:density}, 
assumption (H4) and} the inclusion 
\[
\{u=\psi\}\subseteq\{\nabla u=\nabla\psi\},
\]
consequence of the unilateral obstacle condition $u\geq\psi$ on $\Omega$ and the 
regularity of both $u$ and $\psi$.
Then set $w:=u-\psi$, and note that for all $x$ in $\Omega$ 
{
\begin{multline}\label{e:integral}
  \nabla_\xi F(x,u(x),\nabla u(x))-\nabla_\xi F(x,u(x),\nabla \psi(x))=
  \nabla_\xi F(x,u(x),\nabla w(x)+\nabla\psi(x))-\nabla_\xi F(x,u(x),\nabla \psi(x))\\
  ={\Big(\int_0^1\nabla_\xi^2F\big(x,u(x),\nabla\psi(x)+t\nabla w(x)\big)dt\Big)}
  \nabla w(x)=:{\mathbb A}(x)\nabla w(x).
 \end{multline}}
From \eqref{e:pde3} and \eqref{e:integral}, we conclude that $w$ satisfies \eqref{e:pde2}. 
Moreover, being $u\,,\psi\in C^{1,1}_{loc}(\Omega)$ and 
{$F\in C^{2,1}_{loc}(\Omega\times\R\times\R^n)$}, 
we deduce that item (i) in the statement is satisfied, as well. 
Moreover, for all $x\in\Omega$ and for all $\xi\in K$, $K\subset\R^n$ a compact set, we have
\begin{multline*}
\nu^{-1}(2^{p-2}\wedge 1)|\xi|^2\int_0^1\big(1+|\nabla\psi(x)+t\nabla w(x)|\big)^{p-2}\,dt
\leq \A(x)\xi\cdot\xi\\
={\int_0^1\nabla_\xi^2 F\big(x,u(x),\nabla\psi(x)+t\nabla w(x)\big)\xi\cdot\xi\,dt
\leq\|\nabla_\xi^2 F\|_{L^\infty(K\times B_{r_K}\times B_{r_K},\R^{n\times n})}|\xi|^2,}
\end{multline*}
with $r_K:=\sup_K(|u|+|\nabla\psi|+|\nabla w|)$. The inequality on the left hand side above is an easy 
consequence of the coercivity condition in \eqref{e:acoercive}.
Ellipticity then easily follows if $p\geq 2$, for $p\in(1,2)$ instead we use that $u\,,\psi\in C^{1,1}_{loc}(\Omega)$.
Finally, the upper bound in \eqref{e:Aprop} follows easily in both cases. The conclusion then follows.
\end{proof}

We are ready to prove Theorem~\ref{t:main} as a direct consequence of Theorem~\ref{t:mainFGS} and 
Lemma~\ref{l:tecnico2}.
\begin{proof}[Proof of Theorem~\ref{t:main}]
Existence of solutions to \eqref{e:mainpb2} follows from \cite[Theorem~4.5]{Gi03} thanks to 
the convexity of $\xi\mapsto F(x,z,\xi)$ and the growth conditions \eqref{e:Fgrowth}. The former 
guarantees lower semicontinuity of the associated functional in the weak $W^{1,p}$ topology, the 
latter ensures its coercivity over ${\mathbb{K}_{\psi,g}}$. Therefore, the Direct Method of the Calculus 
of Variations applies.

Moreover, any minimizer $u$ is $C^{1,1}_{loc}(\Omega)$. To this aim, it suffices to note that $u$ 
satisfies the PDE in \eqref{e:pde}, since the derivation of the latter is independent from item (i) 
in (H1). Note that assumption (H2)$^\prime$ corresponds to \eqref{e:Funcvx}. 

{Hence, in view of Lemma~\ref{l:tecnico2},} to conclude the free boundary analysis we only need 
 to check that, locally in $\Omega$, we may apply Theorem~\ref{t:mainFGS} with matrix field $\A$ as above, 
 with 
 \[
{f:=-\div(\nabla_\xi F(x,u,\nabla \psi))+\partial_zF(x,u,\nabla u)},
 \]
with $0$ obstacle and with boundary datum {$g-\psi$. Indeed, thanks to \eqref{e:pde2}, $w=u-\psi$ is the 
minimizer of the quadratic energy
\[
 \EEE[v]=\int_\Omega\big(\A(x)\nabla v(x)\cdot \nabla v(x)+2f(x)\,v(x)\big)\,dx
\]
over $\mathbb{K}_{g-\psi,0}$. In addition, note that $\partial\{w=0\}\cap\Omega=\partial\{u=\psi\}\cap\Omega$.}

 {With the aim of applying Theorem~\ref{t:mainFGS}} we first recall that $\{u=\psi\}\subseteq\{\nabla u=\nabla\psi\}$, being $u\geq\psi$ on $\Omega$. 
 Thus, given $\Omega^\prime\subset\hskip-0.125cm\subset\Omega$ and any $\eps>0$, the set 
 $\Omega_\eps^\prime:=\{0\leq u-\psi<\eps\}\cap\{|\nabla(u-\psi)|<\eps\}\cap\Omega^\prime$ is open and such that
 $\{u=\psi\}\cap\Omega^\prime\subset\Omega_\eps^\prime$ in view of the remark above. Moreover, as
 {$h=-\div(\nabla_\xi F(x,\psi,\nabla \psi))+\partial_zF(x,\psi,\nabla\psi)\geq c_0>0$} (cf. (H4)), we have on 
 $\Omega_\eps^\prime$
{
\begin{multline*}
  f\geq h-\|h-f\|_{L^\infty(\Omega_\eps^\prime)}\\
  \geq c_0-\|\partial_zF(\cdot,\psi,\nabla \psi)-\partial_zF(\cdot,u,\nabla u)\|_{L^\infty(\Omega_\eps^\prime)}
  -\|\div(\nabla_\xi F(\cdot,\psi,\nabla \psi))-\div(\nabla_\xi F(\cdot,u,\nabla \psi))\|_{L^\infty(\Omega_\eps^\prime)}\\
  \geq c_0-\omega_{\partial_zF}(2\eps)-\omega_{\nabla_{x,\xi}^2 F}(\eps)
  -\|\nabla u\|_{L^\infty(\Omega_\eps^\prime,\R^n)}\omega_{\nabla_{z,\xi}^2F}(\eps)\\
  -\eps\|\nabla_{z,\xi}^2F(\cdot,\psi,\nabla\psi)\|_{L^\infty(\Omega_\eps^\prime)}
  -\|\nabla^2\psi\|_{L^\infty(\Omega_\eps^\prime,\R^{n\times n})}\omega_{\nabla_\xi^2F}(\eps),
 \end{multline*}}
denoting with {$\omega_{\vartheta}$} a modulus of continuity of the relevant function {$\vartheta$} 
on $\Omega^\prime$ (recall that $F\in C^{2,1}_{loc}$).
 Therefore, we can choose $\eps>0$ sufficiently small in order to accomplish the condition $f\geq\sfrac{c_0}2>0$
 on $\Omega_\eps^\prime$.
 In addition, $f\in C^{0,\alpha}_{loc}(\Omega)$ by hypotheses (H3), {(H5)} and by Theorem~\ref{t:noptreg}. 
 Hence, all the conditions in the statement of Theorem~\ref{t:mainFGS} are satisfied on the open set 
 $\Omega_\eps^\prime$, thus the conclusions follow straightforwardly.
\end{proof}

\section{Locally coercive vector fields}\label{s:lcvf}

The analysis in Section~\ref{s:cvf} does not cover many cases of interest, most relevantly
that of the area functional where
\[
 F(\xi)=\sqrt{1+|\xi|^2},\qquad
\aaa(\xi)=\nabla F(\xi)=\frac{\xi}{\sqrt{1+|\xi|^2}}.
\]
The latter vector field clearly does not fulfill \eqref{e:acoercive} in (H2)$^{\prime}$ 
being $F$ strictly but not uniformly convex. 
Moreover, for such a vector field also the existence of solutions to the corresponding variational 
inequality is not guaranteed in general and requires additional conditions on the set $\Omega$, on
the obstacle $\psi$ and on the boundary datum $g$ (cf.~\cite[Section~4 of Chapter~IV]{KS}), 
\cite[Chapter~1]{Gi03} and the references therein).
The same considerations hold more generally for \emph{locally coercive} vector fields 
$\aaa$ (cf. \cite[Section~4 of Chapter~IV]{KS} in the autonomous case and Theorem~\ref{t:Ger}
below). 

Assuming a priori the existence of a solution and its global Lipschitz continuity, 
the next result due to Gerhardt implies its global $C^{1,1}$ regularity. 
\begin{theorem}[Theorem~0.1 \cite{Ger85}]\label{t:Ger}
 Let $\Omega$ be of class $C^{3,\alpha}$, for some $\alpha\in(0,1)$, $g\in C^{2,1}(\overline{\Omega})$ 
 and $\psi\in C^{1,1}(\overline{\Omega})$. Let $a_0\in C^{1,1}(\Omega\times\R\times\R^n)$, and assume 
 that $\aaa(\cdot,\cdot,\xi)$ is $C^{1,1}(\Omega\times\R,\R^n)$ for all $\xi\in\R^n$, that $\aaa(x,z,\cdot)$ 
 is $C^{2,1}(\R^n,\R^n)$ for all $(x,z)\in\Omega\times\R$, and that for all $(x,z,\eta)\in\Omega\times\R\times\R^n$ 
  \[
  \partial_\xi\aaa(x,z,\eta)\xi\cdot\xi>0\qquad\text{for all $\xi\neq 0$}.
  \]
 If $u\in C^{0,1}(\Omega)$ is a solution of the variational inequality in \eqref{e:mainpb} over the set 
 \[
 \{v\in C^{0,1}(\Omega):\,v\geq\psi\,\text{ on $\Omega$},\quad v=g\,\text{ on $\partial\Omega$}\}
 \]
 then $u\in C^{1,1}(\Omega)$.
\end{theorem}
Therefore, with Theorem~\ref{t:Ger} at hand, if a locally coercive vector field corresponds to an integrand $F$ 
satisfying hypothesis {(H5)} 
of Theorem~\ref{t:main} we can argue as in Lemma~\ref{l:tecnico2} 
and in the second part of the proof of Theorem~\ref{t:main} itself to conclude the same stratification 
result for the free boundary of a solution $u$. 
Note that, in particular, this claim holds for the area functional in the Euclidean space 
(cf. \cite[Section~5 of Chapter~V]{KS} for the two dimensional case, and \cite{Ca77}).

\subsection{The area functional in a Riemannian manifold}
Similarly, we would like to discuss here the case of the
obstacle problem for the area functional in a Riemannian manifold,
that naturally enters in several geometric applications
(cf., \textit{e.g.}, \cite{MoSp}). Indeed, to the best of our knowledge
a comprehensive stratification result of the free boundary points
in this case has not appeared elsewhere.
Since we are aimed here for a local regularity result,
we assume that
\begin{itemize}
\item[(M1)] our manifold is parametrized
by a single chart ${\Sigma}:= B_{r_0}^{n} \times (-r_0,r_0) \subset \R^{n} \times \R$,
for some $r_0>0$;

\item[(M2)] the metric tensor $g$ satisfies $g(0) = I$ and 
$\nabla g(0) = 0$ (where $\nabla$ denotes the Levi-Civita connection);
 
\item[(M3)] the obstacle $\psi \in C^{1,1}(B_{r_0}^{n},(-r_0,r_0))$
with $\psi(0) = |\nabla \psi(0)| = 0$;
\end{itemize}
We consider the following obstacle problem:
\begin{equation}\label{e:min rieman}
\min_{v\in {\mathbb{K}_{\psi,g}}} \; \textup{vol}_g \big(\graph(v)\big),
\end{equation}
where ${\mathbb{K}_{\psi,g}}:= \big\{v \in C^{0,1}(B_{r_0}^{n},(-r_0, r_0))
\,:\, v\geq \psi, \, v\vert_{\de B_{r_0}^n} = g\big\}$ for
some $g\in C^{0,1}(\de B_{r_0}^n)$ with $g \geq \psi\vert_{\de B_{r_0}^n}$,
$\graph(v) := \big\{ (x,v(x))\,:\, x \in B_{r_0}^n \big\} \subset \R^n \times \R$
and $\textup{vol}_g \big(\graph(v)\big)$ is the area ($n$-dimensional
measure) of the
Lipschitz submanifold associated to the graph of $v$.
In local coordinates, one can express the area of $\graph(u)$
in the following way: let ${G}:B_{r_0}^n \to \Sigma$ be given by
${G}(x) = \big(x, u(x)\big)$ and
\[
J{G}(x) := \sqrt{\textup{det} \big(D{G}(x)^T g({G}(x)) D{G}(x) \big)};
\]
then
\[
\textup{vol}_g \big(\graph(u)\big) = \int_{B_{r_0}^n} J{G}(x)\,dx.
\]
More explicitly, the matrix $M(x) := D{G}(x)^T g ({G}(x)) D{G}(x)$
has entries for $i,j=1,\ldots,n$
\[
M_{ij}(x) := g_{ij}\big(x,u(x)\big) + g_{j(n+1)}\big(x,u(x)\big)\,\de_i u(x)
+ g_{i(n+1)}\big(x,u(x)\big)\,\de_j u(x)+g_{(n+1)(n+1)} \de_i u(x) \,\de_ju(x).
\]
As for the case of a flat metric, the existence of solutions to
\eqref{e:min rieman} is not always
guaranteed and several conditions for it should be verified.
However we do not investigate this problem in the present note,
but we assume that we are given a solution $u\in C^{0,1}(B_{r_0}^{n},(-r_0, r_0))$
and moreover we assume that
\begin{itemize}
\item[(M4)] $u \in C^{1,\alpha}(B_{r_0}^{n},(-r_0, r_0))$ for some $\alpha>0$,
and $u(0) = |\nabla u(0)| = 0$.
\end{itemize}

\begin{remark}
A comment regarding the assumption (M4) is in order.
The natural setting for the study of obstacle problems in Riemannian
manifolds is that of the so called ``parametric minimal surfaces''
theory, \textit{i.e.}~the theory of Caccioppoli sets minimizing
the perimeter among all sets which contain (or are contained in)
a given obstacle.
In this setting the existence issue for the obstacle problem is a simple
consequence of the compactness property of Caccioppoli sets, although in general 
the graphical property would not be ensured.

On the other hand, around points of the free boundary of the
solutions it is simple to check that one can choose normal
coordinates in such a way that hypotheses (M1)--(M4) are matched.
In particular, the hypothesis (M4) is a consequence of the
\textit{almost minimizing} property of the solutions to the parametric obstacle
problem and of a Bernstein theorem (cf. \cite[Section~6.1.2]{MoSp} and
\cite{Tamanini84}), and therefore it is not restrictive to assume it.
\end{remark}

In order to better understand the structure of the area functional,
we can follow the strategy in \cite{MoSp} and look at the first variations of the functional
\begin{equation}\label{e:first variations}
\frac{d}{d\eps}\Big\vert_{\eps = 0^+} \textup{vol}_g \big(\graph(u+\eps \phi)\big)\geq 0,
\end{equation}
for every $\phi \in C^{\infty}_c(B_{r_0}^n)$ such that $\phi\vert_{\Lambda_u} \geq 0$
where $\Lambda_u := \{ u = \psi\}$.
By following the computations in \cite{MoSp} we infer that
the inequality \eqref{e:first variations} reads as
\begin{equation}\label{e:first variations2}
\int_{B_{r_0}^n} \phi\,Lu\, dx \leq 0 \quad \forall \; 
\phi\in C^{\infty}_c(B_{r_0}^n),\; \phi\vert_{\Lambda_u} \geq0,
\end{equation}
where 
\[
L u (x) := \textup{div} \Big( A\big(x, u(x), \nabla u (x)\big)\nabla u(x)
+ b\big(x, u(x), \nabla u (x) \big)\Big) - f\big(x,u(x), \nabla u(x)\big),
\]
and $A$, $b$ and $f$ are given by the following formulas (the Einstein
convention of repeated indices is consistently employed in the sequel):
\begin{itemize}
\item[(1)] $A = (a^{ij})_{i,j=1,\ldots,n} : B_{r_0}^n \times (-r_0,r_0) \times \R^n
\to \R^{n\times n}$ is given by
\[
a^{ij}(x, z, \xi) := g_{(n+1)(n+1)}(x,z) h^{ij}(x, z, \xi),
\]
and $(h^{ij})_{i,j=1,\ldots,n}$ is the inverse of the 
matrix $(h_{ij})_{i,j=1,\ldots,n}$ with
\begin{align*}
h_{ij}(x,z,\xi) &:= g_{ij}(x,z) + \xi_i\,g_{j(n+1)}(x,z) + 
\xi_j\,g_{(n+1)i}(x,z)\\
&\quad  + \xi_i\,\xi_j\,g_{(n+1)(n+1)}(x,z)
\quad\quad\quad \forall\;i,j=1,\ldots, n,
\end{align*}
(note that $(h_{ij})_{i,j=1,\ldots,n}$ is non-singular
for small enough $|x|,|z|,|\xi|$);

\item[(2)] $b=(b^i)_{i=1,\ldots,n} : B_{r_0}^n \times (-r_0,r_0) \times \R^n \to \R^n$
is given by
\[
b^i(x, z, \xi) := g_{j(n+1)}(x,z) h^{ji}(x,z,\xi);
\]

\item[(3)] $f:B_{r_0}^n \times (-r_0,r_0) \times \R^n \to \R$ is given by
\begin{align*}
f(x,z,\xi) & := h^{ij}\,\xi_i\, \Gamma_{(n+1)(n+1)}^k\,g_{jk} +
h^{ij}\,\xi_j\,\xi_i \, \Gamma_{(n+1)(n+1)}^k\,g_{k(n+1)}\\
& \quad + h^{ij}\,\Gamma_{i(n+1)}^k\,g_{jk} + h^{ij}\,\xi_j\,\Gamma_{i(n+1)}^k\,g_{k(n+1)},
\end{align*}
where to simplify the notation we have
written $h^{ij} = {h}^{ij}\big(x,z, \xi\big)$, $g_{ij}= g_{ij}\big(x, z\big)$ and $\Gamma_{ij}^k = 
\Gamma_{ij}^k\big(x,z \big)$ denote the Christoffel symbols.
\end{itemize}

Note that \eqref{e:first variations2} reads as a differential inequality
of the form \eqref{e:mainpb} where
\[
\aaa(x,z,\xi) = A(x,z,\xi)\xi + b(x,z,\xi)
\quad\text{and}\quad
a_0(x,z,\xi) = f(x,z,\xi).
\]
We now verify that there exists $s_0 <r_0$ such that
$\aaa$ and $a_0$ above satisfy the conditions of
Theorem~\ref{t:optreg} as long as $|x|+|z|+|\xi| < s_0$,
\textit{i.e.}~(H1) with (iii)$^\prime$ replacing (iii)
and $p=2$, (H2)$^\prime$ for $p=2$.

For what concerns (H1), we note that $\aaa$ and $a_0$ are smooth functions
in their domains and therefore (i), (ii) and (iii)$^\prime$
clearly follows for $|x|+|z|+|\xi| < s_0$ after choosing $\phi_1$ and
$\phi_2$ suitable constants.

Similarly, the upper bound of (H2)$^\prime$ follows from the regularity
of $\aaa$.
For what concerns the coercivity condition we start estimating as follows
(we write $h^{-1}$ for the inverse of the matrix $h=(h_{ij})$):
\begin{align}\label{e:coerc}
\big( \aaa(z,x,\xi) - \aaa(z,x,\eta) \big) \cdot (\xi - \eta) & =
\big(A(x,z,\xi)\xi - A(x,z, \eta) \eta \big) \cdot (\xi - \eta)\notag\\
&\quad+ \big(b(x,z,\xi) - b(x,z,\eta)\big) \cdot (\xi - \eta)\notag\\
& = g_{(n+1)(n+1)}(x,z) \big(h^{-1}(x,z,\xi)\xi - h^{-1}(x,z, \eta) \eta \big)
\cdot (\xi - \eta)\notag\\
& \quad + g_{j(n+1)}(x,z) \big(h^{ji}(x,z,\xi) - h^{ji}(x,z, \eta) \big)
\cdot (\xi_i - \eta_i).
\end{align}
Next note that, since $g(0) = \textup{I}$,
then for every $\kappa>0$ one can find $s_0$ sufficiently small
such that
\begin{equation}\label{e:secondo add}
\big\vert g_{j(n+1)}(x,z) \big(h^{ji}(x,z,\xi) - h^{ji}(x,z, \eta) \big)
\cdot (\xi_i - \eta_i) \big\vert \leq \kappa \, |\xi-\eta|^2.
\end{equation}
On the other hand, we can estimate the first addendum in \eqref{e:coerc}
in the following way:
\begin{align}\label{e:primo add}
\big(h^{-1}(x,z,\xi)\xi - h^{-1}(x,z, \eta) \eta \big)
\cdot (\xi - \eta) & = h^{-1}(x,z,\xi) (\xi - \eta) \cdot (\xi
-\eta)\notag\\
&\quad + \big(h^{-1}(x,z,\xi) - h^{-1}(x,z, \eta) \big)\, \eta \cdot
(\xi - \eta).
\end{align}
We can use the fact that $h^{-1}(0,0,0)= \textup{I}$ and the regularity
of $h^{-1}$ to get that, if $|x|+|z|+|\xi| < s_0$ for some suitably
small $s_0$, then
\begin{align}\label{e:primo add2}
\big(h^{-1}(x,z,\xi)\xi - h^{-1}(x,z, \eta) \eta \big)
\cdot (\xi - \eta)
& \geq \frac{1}{2} \, |\xi - \eta|^2
- \big\vert h^{-1}(x,z,\xi) - h^{-1}(x,z, \eta) \big \vert\, |\eta|
\,|\xi - \eta|\notag\\
&\geq \Big(\frac{1}{2} - \Lip(h^{-1})\,s_0\Big) \, |\xi - \eta|^2.
\end{align}
Using the fact that $g_{(n+1)(n+1)}(0,0) = 1$, we then conclude the
lower bound in (H2)$^\prime$ by choosing a suitable $s_0$ fulfilling
all the requests above.
Note also that \eqref{e:Tmonotone} is also satisfied because $a_0$
does not depend on $z$.

Therefore, if we assume that {(H3) is satisfied}, in view of (M4) we can apply
Theorem~\ref{t:optreg} to $u\vert_{B_{s_0}^n}$, and deduce that our solution 
$u\vert_{B_{s_0}^n}$ has the optimal regularity $C^{1,1}(B_{s_0}^n)$.

Finally, we can consider the regularity of the free boundary of $u$ in $B_{s_0}^{n}$, which can be now obtained by the use of classical arguments.
Indeed, since now $u$ has second derivatives almost everywhere,
we can also rewrite the operator in the following form
(the convention of summation over repeated indices is used):
\begin{align}\label{e:operatore}
Lu & = c^{ij}\big(x, u(x), \nabla u(x)\big)\,\de_{ij}u + d\big(x, u(x), \nabla u(x)\big),
\end{align}
where
\[
c^{ij}(x, z, \xi) = \de_{\xi_i} \aaa_j (x, z, \xi)
\]
and
\[
d(x,z, \xi) = \textup{div}_x \aaa (x,z, \xi) + \de_{z} \aaa (x, z, \xi) \cdot \xi - a_0 (x, z, \xi).
\]
By a simple manipulation of the equation \eqref{e:pde} it follows then that
\begin{align}\label{e:pde rieman}
- c^{ij}\big(x, \psi(x),\nabla \psi(x)\big)& \de_{ij} \big(u(x) - \psi(x)\big)  = 
\Big(L\psi(x) + 
d\big(x, u(x),\nabla u(x)\big)-d\big(x, \psi(x),\nabla \psi(x)\big)\Big) \chi_{\{u>\psi\}}\notag\\
&\quad +
\Big(c^{ij}\big(x, u(x),\nabla u(x)\big) - c^{ij}\big(x, \psi(x),\nabla \psi(x)\big)\Big) \de_{ij}u(x)\, 
\chi_{\{u>\psi\}}.
\end{align}
Moreover, we also deduce from the regularity of $\aaa$ and $a_0$ that,
up to reducing eventually $s_0$, the function ${w}:=u-\psi$ satisfies the following obstacle problem
\begin{equation}\label{e:non-div}
\A^{ij}(x) \de_{ij}{w}(x) = q(x) \chi_{\{{w}>0\}},
\end{equation}
where the matrix field $\A^{ij}(x) = c^{ij}\big(x, \psi(x),\nabla \psi(x)\big)$ is uniformly elliptic, and
\begin{align*}
q(x) &= -L\psi(x) 
-\Big(d\big(x, u(x),\nabla u(x)\big)-d\big(x, \psi(x),\nabla \psi(x)\big)\Big) \chi_{\{u>\psi\}}\\
&\quad-
\Big(c^{ij}\big(x, u(x),\nabla u(x)\big) - c^{ij}\big(x, \psi(x),\nabla \psi(x)\big)\Big) \de_{ij}u(x)\, \chi_{\{u>\psi\}}.
\end{align*}
{By additionally assuming (H4), we have that $-L\psi(x) \ge c_0>0$ and}
$q >\sfrac{c_0}{2} >0$. Furthermore, if the obstacle $\psi \in C^{2,\alpha}$ 
for some $\alpha>0$ then $q \in C^{0,\alpha}$ (where, for the last claim, the 
Schauder estimates for the second derivatives of $w$ in $\{w>0\}$ are used 
(cf.~\cite[Theorem~6.2]{GT}), and the regularity of $u$ which implies that 
$|\nabla u(x) - \nabla \psi(x)| \leq C\, \textup{dist}(x, \{u=\psi\})$).

Now, by using the regularity results for such obstacle problem in \cite{Ca77, Monneau} we can easily conclude the following final result.

\begin{theorem}\label{t:main2}
Let $(\Sigma, g)$ be a Riemannian manifold satisfying
conditions (M1) and (M2), and let $u$ be satisfying (M4) and be a solution to the
obstacle problem for the area functional with respect to an obstacle 
$\psi \in C^{2,\alpha}(B_{r_0}^{n},(-r_0,r_0))$ {satisfying (M3) and such 
that $-L\psi(x)\ge c_0>0$}.

Then, there exists $s_0>0$ such that
$u\in C^{1,1}(B_{s_0}^n,(-r_0,r_0))$ 
and the free boundary decomposes as $\de \{u = \psi\} \cap B_{s_0}^n = \Reg(u) \cup \Sing(u)$, 
where $\Reg(u)$ and $\Sing(u)$ are called its regular and singular part, respectively. Moreover, 
$\Reg(u)\cap\Sing(u)=\emptyset$ and
\begin{itemize}
\item[(i)] $\Reg(u)$ is relatively open in $\de \{u = \psi\}$
and, for every point $x_0 \in \Reg(u)$, there exist $r=r(x_0)>0$
and $\beta= \beta(x_0) \in (0,1)$ such that
$\Reg(u) \cap B_r(x_0)$ is a $C^{1,\beta}$ submanifold of dimension $n-1$;
\item[(ii)] $\Sing(u) = \cup_{k=0}^{n-1} S_k$, with $S_k$ contained in the union of
at most countably many submanifolds of dimension $k$ and class $C^1$.
\end{itemize}
\end{theorem}

{
\begin{remark}
Recalling that {the operator $L$ is the first variation of the area functional}, the condition 
(H4) can be read as the geometric property of the obstacle $\psi$ of having the mean curvature vector 
``pointing downward'', i.e. on the opposite side with respect to the graph of $u$. 
\end{remark}
}

\section*{Acknowledgments} 
M.F. thanks Guido De Philippis for interesting conversations on the subject of the paper.


\begin{thebibliography}{99}


\bibitem{AndLindShah} Andersson, J.; Lindgren, E.; Shahgholian, H., 
Optimal regularity for the obstacle problem for the p-Laplacian.
J. Diff. Equations 259 (2015), Issue 6, 2167--2179.

\bibitem{Anzellotti83} Anzellotti, G., 
 On the $C^{1,\alpha}$-regularity of $\omega$-minima of quadratic functionals. 
 Boll. Un. Mat. Ital. C (6) 2 (1983), no. 1, 195--212.


{ \bibitem{Blank} Blank, I.
 Sharp results for the regularity and stability of the free boundary in the obstacle problem.
 Indiana Univ. Math. J., 50 (2001), no. 3, 1077--1112.}

\bibitem{BaCa84} Baiocchi, C.; Capelo, A.,
Variational and quasivariational inequalities. Applications to free boundary problems. 
John Wiley \& Sons, Inc., New York, 1984. ix+452 pp.

\bibitem{Brezis72} Brezis, H.R., 
Probl\`emes unilat\'eraux. J. Math. Pures Appl. (9) 51 (1972), 1--168.

{\bibitem{BrezKinder73}
 Brezis, H.R.; Kinderlehrer, D.,
 The smoothness of solutions to nonlinear variational inequalities. 
 Indiana Univ. Math. J. 23 (1973/74), 831--844. 
}

\bibitem{BrezisStampacchia68} Brezis, H.R.; Stampacchia, G., 
Sur la r\'egularit\'e de la solution d'in\'equations elliptiques. 
(French) Bull. Soc. Math. France 96 (1968), 153--180.

\bibitem{Ca77} Caffarelli, L.A., 
The regularity of free boundaries in higher dimensions. 
Acta Math. 139 (1977), no. 3-4, 155--184.

\bibitem{Ca80} Caffarelli, L.A., 
Compactness methods in free boundary problems. 
Comm. Partial Differential Equations 5 (1980), no. 4, 427--448.

\bibitem{Ca98-Fermi} Caffarelli, L.A., 
The obstacle problem revisited. 
Lezioni Fermiane. [Fermi Lectures] Accademia Nazionale dei Lincei, Rome; 
Scuola Normale Superiore, Pisa, 1998. ii+54 pp.

\bibitem{Ca98} Caffarelli, L.A., 
The obstacle problem revisited. 
J. Fourier Anal. Appl. 4 (1998), no. 4-5, 383--402.



{\bibitem{CaffaKind80}
Caffarelli, L.A.; Kinderlehrer, D.,
Potential methods in variational inequalities. J. Analyse Math. 37 (1980), 285--295.
}


{\bibitem{CaffRiv} 
Caffarelli, L.A.; Rivi\'ere, N.M., 
Smoothness and analyticity of free boundaries in variational inequalities.
Ann. Scuola Norm. Sup. Pisa Cl. Sci. (4) 3 (1976), no. 2, 289--310.}


\bibitem{CS} Caffarelli, L.A.; Salsa, S.,
A geometric approach to free boundary problems. Graduate Studies in Mathematics, 68. 
American Mathematical Society, Providence, RI, 2005. x+270 pp.


%


\bibitem{Duvaut-Lions} Duvaut, G.; Lions, J.L., 
Les in\'equations en m\'echanique et en physique, Dunod, Paris, (1972).

\bibitem{FocGelSp15} Focardi, M.;  Gelli, M.S.; Spadaro, E., 
Monotonicity formulas for obstacle problems with Lipschitz coefficients.
Calc. Var. Partial Differential Equations 54 (2015), 1547--1573


{\bibitem{Frehse72}
Frehse, J.,
On the regularity of the solution of a second order variational inequality. 
Boll. Un. Mat. Ital. (4) 6 (1972), 312--315.
}

\bibitem{FrehseMosco82}  Frehse, J.; Mosco, U., 
Irregular obstacles and quasivariational inequalities of stochastic impulse control. 
Ann. Scuola Norm. Sup. Pisa Cl. Sci. (4) 9 (1982), no. 1, 105--157.

\bibitem{Frie} Friedman, A.,
Variational Principles and Free Boundary Problems.
Second edition. Robert E. Krieger Publishing Co., Inc., Malabar, FL, 1988. x+710 pp. 

{
\bibitem{Fu90} Fuchs, M., 
H\"older continuity of the gradient for degenerate variational inequalities.
Nonlinear Anal. TMA 15 (1990), No. 1, 85--100.

\bibitem{FuMin00} Fuchs, M.; Mingione G., 
Full $C^{1,\alpha}$-regularity for free and constrained local
minimizers of elliptic variational integrals with nearly linear growth.
Manuscripta Math. 102 (2000), 227--250.
}

\bibitem{Geraci16} Geraci, F.,
The classical obstacle problem with coefficients in fractional Sobolev spaces.
Preprint 2016.

\bibitem{Ger73} Gerhardt, C., 
Regularity of solutions of nonlinear variational inequalities.
Arch. Ration. Mech. Anal. 52 (1973), 389--393. 

\bibitem{Ger85} Gerhardt, C., 
Global $C^{1,1}$-regularity for solutions of quasilinear variational inequalities.
Arch. Ration. Mech. Anal. 89 (1985), 83--92. 

\bibitem{Giaq81} Giaquinta, M.,
Remarks on the regularity of weak solutions to some variational inequalities.
Math. Z. 177  (1981), 15--31.

\bibitem{GiaqGiu82} Giaquinta, M.; Giusti, E., 
On the regularity of the minima of variational integrals. Acta Math. 148 (1982), 31--46.

\bibitem{GiaqGiu84} Giaquinta, M.; Giusti, E., 
Quasiminima. 
Ann. Inst. H. Poincar\'e Anal. Non Lin\'eaire 1 (1984), no. 2, 79--107.


\bibitem{Gi03} Giusti, E., 
Direct methods in the calculus of variations. 
World Scientific Publishing Co., Inc., River Edge, NJ, 2003. viii+403 pp.

\bibitem{GT}  Gilbarg, D.; Trudinger, N.S., 
Elliptic partial differential equations of second order. 
Reprint of the 1998 edition. Classics in Mathematics. 
Springer-Verlag, Berlin, 2001. xiv+517 pp.

\bibitem{Gustaffson86} Gustafsson, B., 
 A simple proof of the regularity theorem for the variational inequality of the obstacle problem. 
 Nonlinear Anal. 10 (1986), no. 12, 1487--1490. 

\bibitem{HanouzetJoly79} Hanouzet, B.; Joly, J.L.,
 M\'ethodes d'ordre dans l'interpr\'etation de certaines in\'equations variationnelles et applications. 
 (French) J. Funct. Anal. 34 (1979), no. 2, 217--249.

\bibitem{KS}  Kinderlehrer, D.;  Stampacchia, G., 
An introduction to variational inequalities and their applications. 
Pure and Applied Mathematics, 88. Academic Press, Inc. 
[Harcourt Brace Jovanovich, Publishers], New York-London, 1980. xiv+313 pp.

\bibitem{HartmanStampacchia66} Hartman, P.; Stampacchia, G.,
On some non-linear elliptic differential-functional equations.
Acta Math. 115 (1966), 271--310. 

\bibitem{LadUra} Ladyzhenskaya, O.; Uraltseva, N.,
Linear and quasilinear elliptic equations.
Academic Press, New York-London 1968 xviii+495 pp. 

\bibitem{LewyStampacchi69} Lewy, H.; Stampacchia, G., 
 On the regularity of the solution of a variational inequality. 
 Comm. Pure Appl. Math. 22 (1969), 153--188.

\bibitem{Manfredi88} Manfredi, J.J., 
Regularity for minima of functionals with p-growth. 
J. Differ. Equations 76 (1988), 203--212.
 
\bibitem{ManfrediPhD} Manfredi, J.J., 
Regularity of the gradient for a class of nonlinear possibly degenerate elliptic equations. 
PhD. Thesis. University of Washington, St. Louis.


%

\bibitem{MichaelZiemer86} Michael, J. H.; Ziemer, W.P., 
Interior regularity for solutions to obstacle problems. 
Nonlinear Anal. 10 (1986), no. 12, 1427--1448.

\bibitem{MoSp} Mondino, A.; Spadaro, E.,
On a isoperimetric-isodiametric inequality. Preprint (2016):  arXiv:1603.05263.


\bibitem{Monneau} Monneau, R., 
On the number of singularities for the obstacle problem in two dimensions.  
J. Geom. Anal.  13  (2003),  no. 2, 359--389.

\bibitem{Mosco79} Mosco, U., 
Implicit variational problems and quasi variational inequalities. 
Nonlinear operators and the calculus of variations (Summer School, Univ. Libre Bruxelles, Brussels, 1975), 
pp. 83--156. Lecture Notes in Math., Vol. 543, Springer, Berlin, 1976.

\bibitem{MoscoTroianiello73}
Mosco, U.; Troianiello, G. M.,
On the smoothness of solutions of unilateral Dirichlet problems. (Italian summary)
Boll. Un. Mat. Ital. (4) 8 (1973), 57--67. 



\bibitem{PSU}  Petrosyan, A.; Shahgholian, H.; Ural'tseva, N.N., 
Regularity of free boundaries in obstacle-type problems. 
Graduate Studies in Mathematics, 136. American Mathematical Society, 
Providence, RI, 2012. x+221 pp.

\bibitem{Ro} Rodrigues, J.F.,
Obstacle problems in mathematical physics. North-Holland Mathematics Studies, 134. 
Notas de Matematica [Mathematical Notes], 114. North-Holland Publishing Co., Amsterdam, 1987. xvi+352 pp.


\bibitem{Tamanini84} Tamanini, I.,  
Regularity results for almost minimal oriented hypersurfaces in $\R^N$. 
Quaderni del Dipartimento di Matematica dell'Università del Salento (1984).

\bibitem{Troianiello} Troianiello, G.M.,
Elliptic differential equations and obstacle problems. 
The University Series in Mathematics. Plenum Press, New York, 1987. xiv+353 pp. 

%


\bibitem{Uraltseva87}
Ural'tseva, N.N.,
Regularity of solutions of variational inequalities.
Russian Math. Surveys 42:6 (1987), 191--219

\bibitem{Weiss} Weiss, G.S., 
A homogeneity improvement approach to the obstacle problem. 
Invent. Math.  138  (1999),  no. 1, 23--50.

\bibitem{Weiss01} Weiss, G.S., 
An obstacle-problem-like equation with two phases: pointwise regularity of the solution 
and an estimate of the Hausdorff dimension of the free boundary. 
Interfaces Free Bound. 3 (2001), no. 2, 121--128.





\end{thebibliography}
\end{document}